\documentclass[11pt,oneside]{article}

\usepackage{amsmath}
\usepackage{amsthm}
\usepackage{amsfonts}
\usepackage{amssymb}
\usepackage{bbm}
\usepackage{bm}
\usepackage{enumitem}
\usepackage{setspace}
\usepackage[utf8]{inputenc}
\usepackage[english]{babel}

\usepackage[sort,authoryear]{natbib} 
\usepackage{color}

\usepackage{hyperref}
\hypersetup{
   linktoc=page,
   colorlinks=true,
   linkcolor=blue,
   citecolor=blue,
   urlcolor=cyan
}

\usepackage[a4paper,top=1in, bottom=1.6in]{geometry}

\theoremstyle{definition}
\newtheorem{definition}{Definition}[section]
\newtheorem{example}[definition]{Example}

\theoremstyle{plain}
\newtheorem{theorem}[definition]{Theorem}
\newtheorem{prop}[definition]{Proposition}
\newtheorem{cor}[definition]{Corollary}
\newtheorem{lemma}[definition]{Lemma}

\theoremstyle{remark}

\newcommand{\bb}[1]{\bm{#1}}
\newcommand{\C}{\boldsymbol{C}}

\newcommand{\x}{\boldsymbol{x}}
\newcommand{\om}{\boldsymbol{\omega}}

\newcommand{\y}{\boldsymbol{y}}
\newcommand{\h}{\boldsymbol{h}}

\newcommand{\NN}{\mathbb{N}}

\newcommand{\RR}{\ensuremath{\mathbb{R}}}
\newcommand{\EE}{\ensuremath{\mathbb{E}}}
\newcommand{\Var}{\ensuremath{\text{\normalfont Var}}}

\newcommand{\G}{{\boldsymbol \gamma}}
\newcommand{\TG}{{\boldsymbol {\tilde\gamma}}}

\allowdisplaybreaks

\title{Covariance Models for Multivariate Random Fields resulting from Pseudo Cross-Variograms }

\author{Christopher D\"orr\footnote{Institute of Mathematics, University of Mannheim, 68131 Mannheim, Germany. Email: chrdoerr@mail.uni-mannheim.de} {} and Martin Schlather\footnote{Institute of Mathematics, University of Mannheim, 68131 Mannheim, Germany.} }
\date{\today}

\begin{document}
\setstretch{1.25}

\maketitle

\begin{abstract}
So far, the pseudo cross-variogram is primarily used as a tool for the structural analysis of multivariate random fields. Mainly applying recent theoretical results on the pseudo cross-variogram, we use it as a cornerstone in the construction of valid covariance models for multivariate random fields. In particular, we extend known univariate constructions to the multivariate case, and generalize existing multivariate models. Furthermore, we provide a general construction principle for conditionally negative definite matrix-valued kernels, which we use to reinterpret previous modeling proposals.    
\end{abstract}

{\small
  \phantom{\textit{2010 MSC}:} {Secondary 62M30; 62M40} 
}

\section{Introduction}

Multivariate data are of ever-increasing importance in today's world. They usually show
dependencies between the variables and therefore contain additional information to exploit;
so joint instead of separate modeling is needed to use them to full advantage. A popular approach in that regard is to use multivariate random fields, and to describe the dependence structure via a cross-covariance function. This way, cross-covariance functions are applied to problems in various areas, which include atmospheric science \citep{apanasovich2010, li2008, qadir2020flexible}, meteorology \citep{bourotte2016, gneiting2010, cressie2016, genton2015}, oceanography \citep{li2011}, or geology \citep{moreva2016}, for instance.

Constructing valid cross-covariance functions is a challenging task. Several approaches have been proposed, which include latent dimensions \citep{apanasovich2010}, the linear model of coregionalization \citep{goulard1992}, convolution methods \citep{ver1998,majumdar2007}, deformations \citep{sampson1992, vu2021}, multivariate adaptation \citep{gneiting2010, apanasovich2012, moreva2016, porcu2018shkarofsky}, mixtures \citep{porcu2011characterization, bourotte2016}, and a conditional approach \citep{cressie2016}. The benefits and drawbacks of the resulting models are sufficiently well-known. For instance, all referenced models stemming from the multivariate adaptation approach are symmetric, which can lead to inferior predictions, cf.\,\citet{li2011}.

Pseudo cross-variograms are useful quantities for the structural analysis of multivariate random fields. Apart from their usage in multivariate geostatistics, they also appear naturally in exteme value theory in the context of multivariate Brown-Resnick processes \citep{genton2015maxstable}. Only recently, theoretical results on pseudo cross-variograms have been established \citep{doerr2021characterization}, which bridge between pseudo cross-variograms and matrix-valued correlation functions through a matrix-valued version of Schoenberg's theorem \citep{schoenberg1938}. This intimate connection adds another dimension to the range of applications of pseudo cross-variograms, that is, the construction of valid covariance models for multivariate random fields. In this regard, pseudo cross-variograms have already been used in \citet{doerr2021characterization}, \citet{allard2022fully} and \citet{porcu2022criteria} to propose several extensions of Gneiting's popular univariate space-time covariance model \citep{gneiting2002}, thereby meeting one of the requests of \citet{chen2021space} for flexible space-time cross-covariance models.

Our aim here is to further highlight the potential of pseudo cross-variograms for the construction of (asymmetric) cross-covariance models. To this end, we present several selected extensions of univariate constructions found in the literature which can be transferred to the multivariate case via pseudo cross-variograms. We also illustrate that pseudo cross-variograms can be used to further generalize, in some sense, parsimonious cross-covariances.   

In Section~2, we briefly provide the necessary background on pseudo cross-variograms and cross-covariance functions. In Section~3, we present a general construction principle for conditionally negative definite kernels and briefly discuss some existing construction principles for pseudo cross-variograms. In the remaining sections, we present several matrix-valued covariance models, starting with mixture proposals in Section~4, followed by non-stationary models, and models involving derivatives in Sections~5 and 6. Eventually, we present a particular class of infinitely divisible matrix-valued models in Section~7.

\section{Preliminaries} \label{sec:prelim}

A matrix-valued kernel $\C: \RR^d \times \RR^d \rightarrow \RR^{m\times m}$ with $C_{ij}(\x,\y)=C_{ji}(\y,\x)$, $\x,\y \in \RR^d$, $i,j=1,\dots,m$, is called positive definite, if $$\sum_{i=1}^n\sum_{j=1}^n \bb{a_i}^\top \C(\bb{x_i},\bb{x_j})\bb{a_{j}} \ge 0$$  for all $n \in \NN$, $\bb{x_1},\dots,\bb{x_n} \in \RR^d$, $\bb{a_1},\dots,\bb{a_n} \in \RR^m$.  The class of matrix-valued positive definite kernels coincides with the class of cross-covariances for multivariate random fields. 
If the kernel $\C$ only depends on the difference $\x-\y$, then $\C$ is called a positive definite (matrix-valued) function.

For an $m$-variate random field $Z$ such that $Z_i(\x)-Z_j(\y)$ is square integrable for all $\x, \y \in \RR^d$, $i,j=1,\dots,m$, the non-stationary pseudo cross-variogram  $\G: \RR^d \times \RR^d \rightarrow \RR^{m \times m}$  
is defined via
\begin{equation*}
\gamma_{ij}(\x,\y)= \frac12\Var(Z_i(\x)-Z_j(\y)), \quad \x,\y \in \RR^d, \quad i,j=1,\dots,m.
\end{equation*}If $\Var(Z_i(\x + \h) - Z_j(\x))$ does not depend on $\x$ for all $\x, \h\in \RR^d$, $i,j=1,\dots,m$, then the pseudo cross-variogram $\G: \RR^d \rightarrow \RR^{m \times m}$ 
is given via 
\begin{equation*}
\gamma_{ij}(\h)= \frac12\Var(Z_i(\x+\h)-Z_j(\x)), \quad \x,\h \in \RR^d, \quad i,j=1,\dots,m.
\end{equation*} 
The diagonal entries of pseudo cross-variograms are univariate variograms. 
A matrix-valued kernel $\G$ is a non-stationary pseudo cross-variogram, if and only if $\gamma_{ii}(\x,\x)=0$, for all $\x \in \RR^d$, $i=1,\dots,m$, and $\G$ is a conditionally negative definite matrix-valued kernel, i.e.,\,$\gamma_{ij}(\x,\y)=\gamma_{ji}(\y,\x)$, $\x,\y \in \RR^d$, $i,j=1,\dots,m$, and
$$\sum_{i=1}^n\sum_{j=1}^n \bb{a_i}^\top \G(\bb{x_i},\bb{x_j})\bb{a_j} \le 0, 
$$
for all $n \in \NN$, $\bb{x_1},\dots,\bb{x_n} \in \RR^d$, $\bb{a_1},\dots,\bb{a_n} \in \RR^m$ such that $\bb{1_m}^\top\sum_{k=1}^n \bb{a_k} =0$ with $\bb{1_m}:=~(1,\dots, 1)^\top \in \RR^m$, see \citet{doerr2021characterization}. 

For positive definite and conditionally negative definite matrix-valued kernels, we have the following matrix-valued version of Schoenberg's theorem which we repeat separately here for ease of reference.

\begin{theorem} \label{thm:schoenberg}\citep{berg1984harmonic, doerr2021characterization}.
	A kernel $\G: \RR^d \times \RR^d \rightarrow \RR^{m\times m}$ is conditionally negative definite, if and only if $\exp^\ast(-t\G)$ is positive definite for all $t >0$, where $^\ast$ indicates componentwise application of the exponential function.
\end{theorem}

\section{A Construction Principle for Conditionally Negative Definite Kernels} \label{sec:pseudovar}
Examples of conditionally negative definite functions or kernels are desirable from a practical point of view and with regard to the subsequent sections, where pseudo cross-variograms are central ingredients. We first deal with a general construction principle for conditionally negative definite kernels that is essentially a slight reinterpretation of a known relation between positive definite and conditionally negative definite kernels, which we will come back to later. 

\begin{theorem}\label{thm:cndstructure}
	Let $\G: \RR^d \times \RR^d \rightarrow \RR^{m \times m}$ be a matrix-valued conditionally negative definite kernel. Then $\G$ has the form
	\begin{equation} \label{eq:cndstructure}
	\gamma_{ij}(\x,\y)= g_i(\x)+g_j(\y) - C_{ij}(\x,\y), \quad \x,\y \in \RR^d, \quad i,j=1,\dots,m, 
	\end{equation} for some functions $g_i: \RR^d \rightarrow \RR$, $i=1,\dots,m$, and a positive definite kernel $\C$.
	On the other hand, if a function $\G: \RR^d \times \RR^d \rightarrow \RR^{m\times m}$ has the form \eqref{eq:cndstructure}, then $\G$ is conditionally negative definite. In particular, if $\gamma_{ii}(\x,\x)=0$, for all $\x \in \RR^d$, $i=1,\dots,m$, then $\G$ is a non-stationary pseudo cross-variogram.
\end{theorem}     

\begin{proof} 
	Since $\G$ is conditionally negative definite, the kernel $\C: \RR^d \times \RR^d \rightarrow \RR^{m \times m}$ with $$ C_{ij}(\x,\y):=\gamma_{i1}(\x, \bb{x_0}) + \gamma_{j1}(\y,\bb{x_0}) -\gamma_{ij}(\x,\y)-\gamma_{11}(\bb{x_0},\bb{x_0}), \quad i,j=1,\dots,m,$$ is positive definite for $\bb{x_0} \in \RR^d$ \citep{berg1984harmonic,doerr2021characterization}. Choosing $g_i(\x) := \gamma_{i1}(\x,\bb{x_0})-\frac12 \gamma_{11}(\bb{x_0},\bb{x_0})$, $i=1,\dots,m$, and rearranging shows the first part. For the second part, since $-\C(\x,\y)$ is obviously conditionally negative definite, we only need to show that a matrix-valued kernel $\bb{g}: \RR^d \times \RR^d \rightarrow \RR^{m \times m}$ with $g_{ij}(\x,\y)=g_i(\x)+g_j(\y), i,j=1,\dots,m$, is conditionally negative definite, but this is a consequence of Proposition 3.1.9 in \citet{berg1984harmonic}. 
	The specification for non-stationary pseudo cross-variograms follows from Theorem~2.2 in \citet{doerr2021characterization}.
\end{proof}

Theorem~\ref{thm:cndstructure} provides a non-unique representation of conditionally negative definite matrix-valued kernels. It shows that any such kernel is the difference of an additive separable kernel and a positive definite kernel, which is an easily applicable and flexible construction principle at the same time. The functions $g_1, \dots, g_m$ can be chosen appropriately; the positive definite kernel $\C$ can be selected from the current pool of cross-covariances in the literature. In case that the cross-covariances are bounded, such as the ones derived from the deformation approach, positivity or non-negativity constraints can also be easily met by including a conditionally negative definite matrix-valued kernel $(\x,\y) \mapsto c\bb{1_m}\bb{1_m}^\top$ for some appropriate $c \in \RR$, in the additive separable structure. 

A particular example of the construction~\eqref{eq:cndstructure}, which involves non-stationary cross-variograms \citep{myers1982matrix}, is inherent in Remark~2 of \citet{schlather2010}.

\begin{cor} \label{cor:cndstructure} 
	Let $\TG: \RR^d \times\RR^d \rightarrow \RR^{m \times m}$ be a non-stationary cross-variogram. Then the kernel $\G: \RR^d \times \RR^d \rightarrow \RR^{m \times m}$, defined via 
	$$\gamma_{ij}(\x,\y) = \tilde\gamma_{ii}(\x,\bb{0})+\tilde\gamma_{jj}(\y,\bb{0})-(\tilde\gamma_{ij}(\x,\bb{0})+\tilde\gamma_{ij}(\y,\bb{0})-\tilde\gamma_{ij}(\x,\y)),\quad i,j=1,\dots,m,$$ is a non-stationary pseudo cross-variogram.
\end{cor}
\begin{proof}
	It is shown in \citet{schlather2010} that the kernel $$(\x,\y) \mapsto \TG(\x,\bb{0})+\TG(\y,\bb{0})-\TG(\x,\y), \quad \x,\y \in \RR^d,$$ is positive definite for a non-stationary cross-variogram $\TG: \RR^d \times \RR^d \rightarrow \RR^{m \times m}$.
	Since $\tilde\gamma_{ii}$ vanishes on the diagonal for all $i=1,\dots,m$, due to the properties of non-stationary cross-variograms, see \citet{du2012variogram}, for instance, the result follows immediately from Theorem~\ref{thm:cndstructure}. 
\end{proof}

Examples of (non-stationary) cross-variograms can be found in \citet{ma2011class}, \citet{ma2011vector}, \citet{du2012variogram}, \citet{arroyo2017spectral} and \citet{chen2019parametric}, for instance.
Corollary~\ref{cor:cndstructure} sheds some new light on Remark~2 in \citet{schlather2010}. It is shown there via a constructive proof that the matrix-valued kernel $$(\x,\y) \mapsto \left(\exp\left(-\tilde\gamma_{ii}(\x,\bb{0})-\tilde\gamma_{jj}(\y,\bb{0})+\tilde\gamma_{ij}(\x,\bb{0})+\tilde\gamma_{ij}(\y,\bb{0})-\tilde\gamma_{ij}(\x,\y)\right)\right)_{i,j=1,\dots,m},$$ 
$\x,\y \in \RR^d$,
is positive definite. We can now easily recover this result in an alternative manner, by simply combining Theorem~\ref{thm:schoenberg} and Corollary~\ref{cor:cndstructure}. 

The constructive approach of \citet{oesting2017} for modeling pseudo cross-vario\-grams is in similar vein to Theorem~\ref{thm:cndstructure}. Their pseudo cross-variogram proposal results from the sum of a univariate, intrinsically stationary random field and a multivariate stationary random field. This approach was recently taken up in \citet{allard2022fully} and augmented in spirit of the delay model \citep{wackernagel2003multivariate, li2011}. Both approaches are very flexible and allow for different variogram structures on the main diagonals, which makes them an appealing choice. In contrast, the pseudo cross-variogram models in \citet{chen2019parametric} based on latent dimensions have identical diagonal entries, as already pointed out by the authors. This is not favourable for our purposes here, since the pseudo cross-variograms will predominantly encode the multivariate dependence structure.  

There is also a kind of \glqq reverse statement\grqq\ of Theorem~\ref{thm:cndstructure} which solely involves pseudo cross-variograms to create positive definite matrix-valued functions. It is a matrix-valued version of an often considered construction, see \citet{berg1984harmonic}, \citet{ma2004use}, \citet{porcu2011}, or \citet{sasvari2013multivariate}, for instance, which is actually inherent in Equation~(6) in \citet{papritz1993}.

\begin{lemma} \label{lem:kernel} 
	Let $\G: \RR^d \rightarrow \RR^{m \times m}$ be a pseudo cross-variogram. Then the function $\C: \RR^d \rightarrow \RR^{m \times m}$, defined via
	
	$$\C(\h)= \G(\h + \bb{z}) + \G(\h-\bb{z}) - 2\G(\h) $$
	is a cross-covariance function for $\bb{z} \in \RR^d$.
\end{lemma}

\begin{proof}
	Let $\bb{Z}$ be an $m$-variate random field with pseudo cross-variogram $\G$. Then the random field $\bb{Y}(\x)=\bb{Z}(\x+\bb{z})-\bb{Z}(\x)$, $\x, \bb{z} \in \RR^d,$ has the desired cross-covariance structure.
\end{proof}

By means of Lemma~\ref{lem:kernel}, we obtain the following matrix-valued version of Theorem~2 of \citet{ma2004use}. 

\begin{prop}
	Let $\G: \RR^d \rightarrow \RR^{m \times m}$ be a pseudo cross-variogram. Then the function $\C: \RR^d \rightarrow \RR^{m \times m}$, defined via
	$$C_{ij}(\h) = \frac{(1 + \gamma_{ij}(\h + \bb{z}))(1 + \gamma_{ij}(\h - \bb{z}))  }{ (1 + \gamma_{ij}(\h))^2} + c, \quad c \ge -1, \quad i,j=1,\dots,m,$$ 
	is positive definite for any $\bb{z} \in \RR^d$.
\end{prop}

\begin{proof}
	The function $t \mapsto \log(1+t), t \ge 0$, is the continuous extension of a Bernstein function which vanishes at zero \citep{schilling2012}. Since Bernstein functions vanishing at zero operate on pseudo cross-variograms \citep{berg1984harmonic, doerr2021characterization}, the function $$\h \mapsto \left(\log(1+\gamma_{ij}(\h))\right)_{i,j=1,\dots,m}, \quad \h \in \RR^d,$$ is again a pseudo cross-variogram. 
	
	Consequently, the function $$\h \mapsto  \left(\log(1+\gamma_{ij}(\h+\bb{z}))+\log(1+\gamma_{ij}(\h-\bb{z}))-2\log(1+\gamma_{ij}(\h))\right)_{i,j=1,\dots,m}, \h \in \RR^d,$$ is positive definite for any $\bb{z} \in \RR^d$ due to Lemma~\ref{lem:kernel}. Applying Theorem~1 in \citet{schlather2010} shows that $$ \h \mapsto \left(\frac{(1 + \gamma_{ij}(\h + \bb{z}))(1 + \gamma_{ij}(\h - \bb{z}))  }{ (1 + \gamma_{ij}(\h))^2} -1\right)_{i,j=1,\dots,m}, $$ is a positive definite matrix-valued function. Since $ \h \mapsto \bb{1_m1_m}^\top$ is positive definite, and the set of positive definite matrix-valued functions forms a convex cone, the function $\C$ is positive definite as well. 
\end{proof}

\section{Multivariate Space-Time Covariance Models}

Covariance models for multivariate spatio-temporal random fields are special positive definite matrix-valued functions on $\RR^d \times \RR$. In the following, we focus on the construction of valid multivariate models based on univariate ones. For properties of the respective univariate models and motivating arguments for the constructions, we refer to the corresponding literature. Since most of the constructions in the sequel are not only valid on $\RR^d \times \RR$, but also on $\RR^d \times \RR^k$ with $k \ge 1$, we also present them in their general forms.       

Mixtures of separable space-time models are a simple and often used approach to construct non-separable space-time  univariate covariance functions, i.e.,\,covariance functions which allow for space-time interactions. 
Recalling that the componentwise product of positive definite matrix-valued functions is again positive definite, see \citet{schlather2010}, for instance, we can immediately generalize Theorem~3 in \citet{ma2003spatio}. 

\begin{theorem} \label{mixture} 
	Let $\G^S: \RR^d \rightarrow \RR^{m \times m}$ and $\G^T: \RR^k \rightarrow \RR^{m \times m}$  be two pseudo cross-variograms. Let $\mathcal{L}_{ij}$, $i,j=1,\dots,m$, be the two-dimensional Laplace transform of a finite, for $i\neq j$ potentially signed measure $\mu_{ij}$ on $[0,\infty)^2$, i.e.,\,$$\mathcal{L}_{ij}(x,y)=\int_{[0,\infty)^2} \exp\left(-vx - wy \right) d\mu_{ij}(v,w), \quad x,y \ge 0.$$ Assume that $\mu_{ij}$ has a density $f_{ij}$ with respect to Lebesgue measure such that the matrix of densities $\left(f_{ij}(v,w)\right)_{i,j=1,\dots,m}$ is symmetric and positive semi-definite for all $v,w\ge 0$. Then the function $\C: \RR^d\times \RR^k \rightarrow \RR^{m \times m}$ defined via
	$$C_{ij}(\h,\bb{u})= \mathcal{L}_{ij}\left(\G_{ij}^S(\h),\G_{ij}^T(\bb{u})\right), \quad i,j=1,\dots,m, $$ is positive definite.
\end{theorem}

\begin{proof}
We follow the proof in \citet{ma2003spatio}. We have	
	\begin{align*}
	C_{ij}(\h,\bb{u}) &= \int_{[0,\infty)^2} \exp\left(-v\gamma^S_{ij}(\h) - w\gamma^T_{ij}(\bb{u}) \right) d\mu_{ij}(v,w)\\
				&= \int_{[0,\infty)^2} \exp\left(-v\gamma^S_{ij}(\h)\right) \exp\left(-w\gamma^T_{ij}(\bb{u})\right)f_{ij}(v,w)   d(v,w).
	\end{align*}

Due to Theorem~\ref{thm:schoenberg}, $\exp^\ast\left(-v\G^S\right)$ and $\exp^\ast\left(-w\G^T\right)$ are positive definite functions for all $v,w \ge 0$; so is their componentwise product, as already mentioned. Since $(f_{ij}(v,w))_{i,j=1,\dots,m},$ is positive semi-definite for all $v,w >0$, $\C$ is a mixture of positive definite functions, hence positive definite itself.   
\end{proof} 
Examples of two-dimensional Laplace transforms can be found in \citet{ma2003spatio}. Other specific choices in the above theorem give matrix-valued versions of some proposals for univariate space-time covariance models. 
Choosing a nested structure $\G^S(\h) = \sum_{i=1}^d \bb{\gamma_i}^{S}(h_i)$ for pseudo cross-variograms $\G_i^S$ on $\RR$ leads to a space-time cross-covariance function in light of Proposition 1 in \citet{porcu2007covariance}. We also obtain a matrix-valued version of Proposition 3.1 in \citet{fonseca2011general}. 

\begin{cor}
	Let $\G^S: \RR^d \rightarrow \RR^{m \times m}$ and $\G^T: \RR^k \rightarrow \RR^{m \times m}$  be two pseudo cross-variograms. Let $\mathcal{L}_0, \mathcal{L}_1, \mathcal{L}_2$ denote the Laplace transforms of some independent non-negative random variables $X_0, X_1, X_2$. Then the function $\C: \RR^d \times \RR^k \rightarrow \RR^{m \times m}$ with
	$$C_{ij}(\h,\bb{u})= \mathcal{L}_0\left(\gamma^S_{ij}(\h)+ \gamma_{ij}^T(\bb{u})\right) \mathcal{L}_1\left(\gamma_{ij}^S(\h)\right) \mathcal{L}_2\left(\gamma_{ij}^T(\bb{u})\right), \quad i,j=1,\dots,m,$$ is a cross-covariance function.
\end{cor}
\begin{proof}
	For all $i,j=1,\dots,m$, consider the Laplace transform of the random vector $(V,W)^\top=(X_0+X_1,X_0+X_2)^\top$ in Theorem~\ref{mixture}.
\end{proof}

For instance, let $X_0,X_2$ be Gamma distributed and let $X_1$ be generalized inverse Gaussian distributed. Then we get the following multivariate version of Theorem~3.1 in \citet{fonseca2011general}.

\begin{cor} \label{cor:fonseca}
  Let $\G^S: \RR^d \rightarrow \RR^{m \times m}$, $\G^T: \RR^k \rightarrow \RR^{m \times m}$ be two pseudo cross-variograms. Then the function $\C: \RR^d \times \RR^k \rightarrow \RR^{m \times m}$ with 
  \begin{eqnarray*}
    C_{ij}(\h,\bb{u})
    &=&\left(1+\frac{\gamma^S_{ij}(\h) +
        \gamma^T_{ij}(\bb{u})}{a_0}\right)^{-\lambda_0}\left(1+\frac{\gamma^S_{ij}(\h)}{a_1}\right)^{-\lambda_1/2}
        \left(1+\frac{\gamma^T_{ij}(\bb{u})}{a_2}\right)^{-\lambda_2} \times
        \\&&\displaystyle
        \frac{K_{\lambda_1}\left(2\sqrt{\left(a_1+\gamma_{ij}^S(\h)\right)\delta}\right)}{K_{\lambda_1}\left(2\sqrt{a_1\delta}\right)
        }
        ,  
  \end{eqnarray*}
$i,j=1,\dots,m$, is positive definite for $a_0,a_1,a_2, \lambda_0, \lambda_1, \lambda_2, \delta >0$, where $K_{\lambda_1}$ denotes the modified Bessel function of the second kind of order $\lambda_1$.
\end{cor}

In the above corollaries, we chose $f_{ij}= f$, $i,j=1,\dots,m$, for a suitable non-negative function $f$ as density matrix of the two-dimensional Laplace transforms. For this particular choice, the matrix $\left(f_{ij}(u,v)\right)_{i,j=1,\dots,m}$ is positive semi-definite for all $u,v\ge 0$ as required. Ensuring positive semi-definiteness of the matrix $\left(f_{ij}(u,v)\right)_{i,j=1,\dots,m}$ for all $u,v$ is the main challenge in Theorem~\ref{mixture}. There are proposals in the geostatistical literature involving product separable functions, for instance, i.e.,\,functions of the form $g_{ij}(\h)=\sqrt{g_i(\h)g_j(\h)}$, $\h \in \RR^d$, $i,j=1,\dots,m$, which form positive semi-definite matrices for all $\h$, see \citet{allard2022fully}. They can be found in \citet{qadir2021semiparametric} or, as a starting point, in \citet{allard2022fully}, for instance. The latter point out that the product separable structure may lead to weaker cross-correlations \citep{allard2022fully}.
Another interesting option is to use the well-known result that a twice continuously differentiable function on an open and convex domain is convex, if and only if its Hessian matrix is positive semi-definite on its entire domain, see \citet{BoydStephenP2009Co}, for instance. Here, finding appropriate convex functions which lead to appealing and closed-form cross-covariance functions, seems to be the main obstacle. A toy example to illustrate this point is the following.

\begin{example}
	Let $f: (0,\infty)^2 \rightarrow \RR$ with $f(v,w)= \frac{v^2}{w}$. Then $f$ is convex on $(0,\infty)^2$, and its Hessian matrix is given by
	$$  \nabla\nabla^\top f(v,w)= \left(\begin{array}{c c}
	2/w & -2v/w^2 \\ -2v/w^2  & 2v^2/w^3  
	\end{array}\right), $$ see \citet[p. 73]{BoydStephenP2009Co}.
	
	Choosing $\mu_{ij}(dv,dw) = \left(\nabla\nabla^\top f(v,w)\right)_{ij}\bb{1}_{(1,2)^2}(v,w) dvdw$, $i,j=1,\dots,m$, gives the two-dimensional Laplace transforms 
	\begin{eqnarray*}
		\mathcal{L}_{11}(x,y) &=& \displaystyle{\frac{2e^{-2x}(e^x - 1)(\mathrm{Ei}(-2y)-\mathrm{Ei}(-y))}{x}}, \\
		\mathcal{L}_{12}(x,y) &=&\displaystyle{\frac{e^{-2(x+y)}(-3xye^{2(x+y)}+2e^{x+y}-2e^y-2e^x+2)}{2xy}},\\
		\mathcal{L}_{22}(x,y) &=& \displaystyle{\frac1{4x^3}(-4x(x+1)+e^x(x(x+2)+2)-2)\times}\\ &&\displaystyle{(4e^{2y}y^2(\mathrm{Ei}(-2y)-\mathrm{Ei}(-y))+2y-4e^y(y-1)-1)e^{-2(y+x)}},
	\end{eqnarray*}
	$x,y \neq 0$, where $\mathrm{Ei}$ denotes the exponential integral function.
\end{example}

In similar fashion to Theorem~\ref{mixture}, the results in \citet{ma2003families} can be extended to the multivariate case by using a pseudo cross-variogram instead of a variogram. Exemplarily, we present a few multivariate versions of the models.
 
 \begin{prop} \label{prop:maternmixture}
 	Let $\G^S: \RR^d \rightarrow \RR^{m \times m}$ and $\G^T: \RR^k \rightarrow \RR^{m \times m}$ be two pseudo cross-variograms. Let $\nu_{ij}=\frac{\nu_{ii}+\nu_{jj}}2$, $i,j=1,\dots,m$, with $\nu_{11},\dots, \nu_{mm} >0$. Then the function $\C: \RR^d \times \RR^k \rightarrow \RR^{m \times m}$ with 
 	$$C_{ij}(\h,\bb{u})= \left(\frac{\gamma_{ij}^S(\h)}{1+ \gamma_{ij}^T(\bb{u})}\right)^{\nu_{ij}/2}K_{\nu_{ij}}\left(\sqrt{\gamma_{ij}^S(\h)(1+\gamma_{ij}^T(\bb{u}))}\right), \quad i,j=1,\dots,m,$$ is positive definite.
 \end{prop}
 
 \begin{proof} We follow the ideas in \citet{ma2003families}. Let $\varepsilon >0$. We have
\begin{eqnarray} 
 	\lefteqn{\int_0^\infty
  2^{\nu_{ij}-1}\Gamma(\nu_{ij})\exp\left(-\frac{\varepsilon +
  \gamma^S_{ij}(\h)}{4\omega}-\gamma^T_{ij}(\bb{u})\omega\right)\frac1{\Gamma(\nu_{ij})}\omega^{\nu_{ij}-1}\exp(-\omega)d\omega}
  \nonumber \qquad
  \\
 		&=& \int_0^\infty 2^{\nu_{ij}-1}\omega^{\nu_{ij}-1}\exp\left(-\frac{\varepsilon + \gamma^S_{ij}(\h)}{4\omega}-(1+\gamma^T_{ij}(\bb{u}))\omega\right)d\omega \nonumber \\
 	&=&	\left(\frac{\varepsilon + \gamma_{ij}^S(\h)}{1+ \gamma_{ij}^T(\bb{u})}\right)^{\nu_{ij}/2}K_{\nu_{ij}}\left(\sqrt{\left(\varepsilon+\gamma_{ij}^S(\h)\right)\left(1+\gamma_{ij}^T(\bb{u})\right)}\right) \label{eq:maternmixture}
 	\end{eqnarray} 
 	due to Formula~(3.471.9) in \citet{gradshteyn2014}. Since $$(\h, \bb{u}) \mapsto \frac1{4\omega}\left(\varepsilon\bb{1_m}\bb{1_m}^\top + \G^S(\h)\right) + \omega\G^T(\bb{u})$$ is a pseudo cross-variogram for all $\omega >0$, the function $$(\h,\bb{u}) \mapsto \left(\exp\left(-\frac{\varepsilon + \gamma^S_{ij}(\h)}{4\omega}-\gamma^T_{ij}(\bb{u})\omega\right)\right)_{i,j=1,\dots,m}$$ is positive definite due to Theorem~\ref{thm:schoenberg}. Further, the matrix $\left((2\omega)^{\nu_{ij}-1}\exp(-\omega)\right)_{i,j=1,\dots,m}$ is positive semi-definite for all $\omega >0$ by construction. Hence, Equation~\eqref{eq:maternmixture} is positive definite as mixture of positive definite functions. Letting $\varepsilon \rightarrow 0$ gives the result.
 \end{proof}
 
 \begin{prop} \label{prop:gaussianextended}
 	Let $\G^1,\dots, \G^n: \RR^k \rightarrow \RR^{m \times m}$ be pseudo cross-variograms. Let $\bb{\Sigma}$ be a symmetric, positive definite matrix, and let $\bb{\Sigma_1}, \dots, \bb{\Sigma_n}$ be symmetric and positive semi-definite matrices. Define $$\bb{A_{ij}}(\bb{u}):=\bb{\Sigma}+\sum_{\ell=1}^n\G^{\ell}_{ij}(\bb{u})\bb{\Sigma_{\ell}}, \quad i,j=1,\dots,m.$$ Then the function $\C: \RR^d \times \RR^k \rightarrow \RR^{m \times m}$ with 
 	$$C_{ij}(\h,\bb{u}) = \left\lvert\bb{A_{ij}}(\bb{u})\right\rvert^{-1/2}\exp\left(-\frac12 \h^\top  \bb{A_{ij}}(\bb{u})^{-1}\h\right), \quad i,j=1,\dots,m, $$ is positive definite. 
 \end{prop}

\begin{proof}
	Again, we follow the ideas in \citet{ma2003families}. We have for a constant $c>0$ that
	\begin{align*}
	C_{ij}(\h,\bb{u})&= c\left\lvert\bb{A_{ij}}(\bb{u})\right\rvert^{-1/2} \int_{\RR^d} \exp(\imath \om^\top\h ) \left\lvert\bb{A_{ij}}(\bb{u})^{-1} \right\rvert^{-1/2}\exp\left(-\frac12 \om^\top \bb{A_{ij}}(\bb{u}) \om \right)  d\om \\
	&= c \int_{\RR^d} \cos(\om^\top\h )\exp\left(-\frac12 \om^\top \bb{A_{ij}}(\bb{u}) \om \right)  d\om, \quad i,j=1,\dots,m.
	\end{align*}
	For any $\om \in \RR^d$, the function $\bb{u} \mapsto \left(\exp\left(-\frac12\sum_{\ell=1}^n\G^{\ell}_{ij}(\bb{u})\om^\top\bb{\Sigma_{\ell}}\om\right)\right)_{i,j=1,\dots,m}$ is positive definite due to Theorem~\ref{thm:schoenberg}; so is $ \h \mapsto \cos(\om^\top \h)$. Hence, $\C$ is a mixture of positive definite functions and hence positive definite itself.
\end{proof}

\begin{prop} \label{prop: lagrangianmixture}
	Let $\G^1, \dots, \G^n$, $\bb{\Sigma}, \bb{\Sigma_1},\dots, \bb{\Sigma_n}$, and $\bb{A_{ij}}$, $i,j=1,\dots,m$, be as in Proposition~\ref{prop:gaussianextended} with $k=1$. Let $\bb{\theta} \in \RR^d$, and let $\mathcal{L}_{ij}$ denote the Laplace transform of a finite, for $i\neq j$ potentially signed measure $\mu_{ij}$ on $[0,\infty)$, $i,j=1,\dots,m$. Assume that $\mu_{ij}$ has a density $f_{ij}$ with respect to Lebesgue measure such that the matrix of densities $\left(f_{ij}(\omega)\right)_{i,j=1,\dots,m}$ is symmetric and positive semi-definite for all $\omega\ge 0$. Then the function $\C: \RR^d \times \RR \rightarrow \RR^{m \times m}$ with
	$$C_{ij}(\h,u) = \left\lvert\bb{A_{ij}}(u)\right\rvert^{-1/2}\mathcal{L}_{ij}\left(\frac12(\h + \bb{\theta}u)^\top \bb{A_{ij}}(u)^{-1}  (\h + \bb{\theta}u)\right),$$ $i,j=1,\dots,m$, is positive definite.
\end{prop}
\begin{proof}
	The function $$(\h,u) \mapsto \left(\left\lvert \bb{A_{ij}}(u)\right\rvert^{-1/2}\exp\left(-\frac12 \omega(\h+\bb{\theta}u)^\top \bb{A_{ij}}(u)^{-1}(\h+\bb{\theta}u)\right)     \right)_{i,j=1,\dots,m} $$ is positive definite for all $\omega \ge 0$ due to Proposition~\ref{prop:gaussianextended}, cf.\,also the univariate version in \citet{ma2003families}. Hence, $\C$ is a mixture of positive definite functions and thus positive definite itself. 
\end{proof}

The cross-covariance function in Proposition~\ref{prop: lagrangianmixture} resembles a model from the Lagrangian framework \citep{salvana2020}, but we do not start from a purely spatial cross-covariance function here. 
Similarly, we can obtain a multivariate analogue of model (19) in \citet{porcu2006nonseparable}, which is based on the Lagrangian framework.

 \begin{prop}
 	Let $\G^1: \RR^{d_1} \rightarrow \RR^{m \times m}$, $\G^2: \RR^{d_2} \rightarrow \RR^{m \times m}$ be purely spatial pseudo cross-variograms, and let $\mathcal{L}$ denote the two-dimensional Laplace transform of a random vector. Let $d_1, d_2 \in \NN$ and $d=d_1+d_2$, and let $\bb{V}=(\bb{V_1},\bb{V_2})$ be a random vector in $\RR^{d_1}\times \RR^{d_2}$. Then the function $\C: \RR^d \times \RR \rightarrow \RR^{m \times m}$ defined via 
 	$$ C_{ij}(\h,u) = \EE_{\bb{V_1},\bb{V_2}} \mathcal{L}\left(\gamma^1_{ij}(\bb{h_1}-\bb{V_1}u),\gamma^2_{ij}(\bb{h_2}-\bb{V_2}u) \right), \h=(\bb{h_1},\bb{h_2}) \in \RR^d, u \in \RR, $$
 	$i,j=1,\dots,m,$ is a space-time cross-covariance function.
 \end{prop}
 
 \begin{proof}
 	Since $\G^1$ and $\G^2$ are purely spatial pseudo cross-variograms, the functions $\tilde\G^i(\bb{h_i},u)=\G^i(\bb{h_i}-\bb{v_i}u)$, $\bb{h_i}, \bb{v_i} \in \RR^{d_i}$, $i=1,2$, are pseudo cross-variograms in $\RR^{d_i} \times \RR$. The assertion then follows with similar arguments as in the proof of Theorem~\ref{mixture}. 
 \end{proof}
 
 \section{Non-stationary spatial models}
 
 So far, we have only used pseudo cross-variograms, resulting in stationary models. 
 But since Theorem~\ref{thm:schoenberg} holds for conditionally negative definite kernels in general, there is also a non-stationary version of Theorem~\ref{mixture}, for instance.
 Besides, non-stationary models can be derived from stationary ones. In the univariate case, \citet{stein2005jasa} advocates the use of stationary models for the construction of non-stationary ones. In the following, we combine both ideas, using stationary models alongside conditionally negative definite kernels to generalize a matrix-valued positive definite kernel in \citet{kleiber2015}.  
 
 \begin{prop}\label{porcukleibernichtstat}
 	Let $\G: \RR^d \times \RR^d \rightarrow (0,\infty)^{m \times m}$ be a positive, conditionally negative definite matrix-valued kernel. Then the matrix-valued kernel $\C: \RR^d \times \RR^d \rightarrow \RR^{m \times m}$, defined via 
 	\begin{equation} \label{eq:porcukleibernichtstat}
 	C_{ij}(\x,\y) =
 	s^{\nu+1}B(\gamma_{ij}(\x,\y)+1,\nu+1)\Psi_{\nu+\gamma_{ij}(\x,\y)+1}\left(\frac{\lVert \x-\y \rVert}{s}\right), \quad \x,\y \in \RR^d,  
 	\end{equation} 
 	$i,j =1,\dots,m$, with $\Psi_{\nu}(t) := (1-t)_{+}^\nu :=\max\{1-t,0\}^\nu$, and $B$ denoting the Beta function,
  is positive definite for $s>0$, $\nu \ge \frac{d+1}2$, and compactly supported. 
 \end{prop}
 
The cross-covariance model~\eqref{eq:porcukleibernichtstat} has been originally formulated in \citet{kleiber2015} for $\gamma_{ij}(\x,\y) = (\gamma_{i}(\x)+\gamma_j(\y))/2, \x,\y\in\RR^d, i,j=1,\dots,m$, for positive valued functions $\gamma_{i}$, $i=1,\dots,m$, i.e.,\, for a purely additive separable, positive, conditionally negative definite kernel. Proposition~\ref{porcukleibernichtstat}, however, shows that we can replace this specific kernel by a general one, thereby allowing for interactions between different locations and introducing more flexibility.

\begin{proof} 
 We follow the ideas in \citet{kleiber2015}, and adjust at one place. Consider the matrix-valued kernel $$C_{ij}(\x,\y)= \int_0^\infty \Psi_\nu\left(\frac{\lVert \x-\y\rVert }{t} \right) t^\nu \left(1-\frac{t}{s}\right)_{+}^{\gamma_{ij}(\x,\y)} dt, \quad \x,\y \in \RR^d, \quad i,j=1,\dots,m.$$	

The function $(\x,\y) \mapsto \Psi_\nu\left(\frac{\lVert \x-\y\rVert }{t} \right)$ is positive definite for $t>0$, $\nu \ge \frac{d+1}2$ \citep{gneiting2002compactly}. Since $\left(1-{t/s}\right)_{+}^{\gamma_{ij}(\x,\y)} = 0$ for $t \ge s$, and
$$ \left(1-\frac{t}{s}\right)_{+}^{\gamma_{ij}(\x,\y)} = \exp\left(\log\left(1-\frac{t}{s}\right)      \gamma_{ij}(\x,\y)  \right)$$  with $ - \log\left(1-{t/s}\right) >0$ for $t \in (0,s)$, the matrix-valued kernel $$(\x,\y) \mapsto \left(\left(1-\frac{t}{s}\right)_{+}^{\gamma_{ij}(\x,\y)}\right)_{i,j=1,\dots,m}, \quad \x,\y \in \RR^d, $$ is positive definite for all $s,t >0$ due to Theorem~\ref{thm:schoenberg}. Theorem~1 in \citet{porcu2011characterization} entails that the matrix-valued kernel $\C$ as defined above is a matrix-valued positive definite function. The same calculations as in \citet{kleiber2015} eventually show that $\C$ equals Equation \eqref{eq:porcukleibernichtstat}.
\end{proof}

Additive separable conditionally negative definite kernels are also basic ingredients in the quasi-arithmetic constructions in Theorem~1 and Theorem~5 in \citet{kleiber2015} which is intimately connected with Theorem~1 in \citet{kleiber2012nonstationary}. Both fall into the category of covariance kernels of the form
\begin{eqnarray*}
  (\x,\y) &\mapsto& \frac{\bb{\sigma_i}(\x)\bb{\sigma_j}(\y)}{\lvert
            \bb{\Sigma_{ij}}(\x,\y)\rvert^{1/2}}\int_0^\infty
                    \exp\left(-t(\x-\y)^\top\bb{\Sigma_{ij}}(\x,\y)^{-1}(\x-\y)\right)
                    g_{ij}(t,\x,\y)d\mu(t),
\end{eqnarray*}
for a matrix-valued
kernel $\bb{g}(t,\x,\y)=
\left(g_{ij}(t,\x,\y)\right)_{i,j=1,\dots,m}$, which is positive
definite for all $t>0$, and 
 positive functions $\sigma_1,\dots,\sigma_m$. 
Specifying $\bb{g}$ via Theorem~\ref{thm:schoenberg} gives the following non-stationary covariance model which allows for a nice interpretation thereafter.

\begin{prop} \label{paciorekconstruction} 
  Let $\G: \RR^d \times \RR^d \rightarrow \RR^{m \times m}$ be a conditionally negative definite kernel. 
  Let $\mu$ denote a measure on $(0,\infty)$, and let $\bb{\Sigma_i}, \bb{\Sigma_j}: \RR^d \rightarrow \RR^{m \times m}$ be functions whose values are symmetric and positive definite matrices, and denote $\bb{\Sigma_{ij}}(\x,\y):=\left(\bb{\Sigma_i}(\x) + \bb{\Sigma_j}(\y)\right)/2$, $i,j=1,\dots,m$. Then the kernel $\C: \RR^d \times \RR^d \rightarrow \RR^{m \times m}$, defined via
  \begin{eqnarray*}
    C_{ij}(\x,\y)
    &=&
        \frac{\lvert\bb{\Sigma_i}(\x)\lvert^{1/4}\lvert\bb{\Sigma_j}(\y)\lvert^{1/4}}{\lvert        \bb{\Sigma_{ij}}(\x,\y)\rvert^{1/2}}
        \times
    \\&&\int_0^\infty \exp\left(-t\left((\x-\y)^\top\bb{\Sigma_{ij}}(\x,\y)^{-1}(\x-\y)+\gamma_{ij}(\x,\y)\right)\right)d\mu(t),
  \end{eqnarray*}
  $\x,\y \in \RR^d$, $i,j=1,\dots,m$, is positive definite, provided the integral exists.
\end{prop}  

\begin{proof}
	The proof of Theorem~1 in \citet{kleiber2012nonstationary}
        shows that
        $$(\x,\y) \mapsto
        \left(\frac{\lvert\bb{\Sigma_i}(\x)\lvert^{1/4}\lvert\bb{\Sigma_j}(\y)\lvert^{1/4}}{\lvert
            \bb{\Sigma_{ij}}(\x,\y)\rvert^{1/2}}
          \exp\left(-t(\x-\y)^\top\bb{\Sigma_{ij}}(\x,\y)^{-1}(\x-\y)\right)\right)_{i,j=1,\dots,m},
        $$
        $\x,\y \in \RR^d$,
        is a positive definite matrix-valued kernel for all $t>0$. Hence, the kernel
	$$(\x,\y) \mapsto
        \left(\frac{\lvert\bb{\Sigma_i}(\x)\lvert^{1/4}\lvert\bb{\Sigma_j}(\y)\lvert^{1/4}}{\lvert
            \bb{\Sigma_{ij}}(\x,\y)\rvert^{1/2}}
          e^{-t\left((\x-\y)^\top\bb{\Sigma_{ij}}(\x,\y)^{-1}(\x-\y)+\gamma_{ij}(x,y)\right)}\right)_{i,j=1,\dots,m},$$
        is also positive definite for all $t>0$ due to Theorem~\ref{thm:schoenberg} as Schur product of positive definite kernels. Since mixtures of positive definite kernels are again positive definite, the assertion follows.  
\end{proof}

\begin{example} \label{ex:maternnonstat} 	 
	Let the conditionally negative definite kernel $\G$ be of the form \begin{equation} \label{eq:maternnonstatcnd}
	 \gamma_{ij}(\x,\y)= \frac{\log(t)}t\frac{\nu_i(\x)}2+\frac{\log(t)}t\frac{\nu_j(\y)}2-G_{ij}(\x,\y), \quad i,j=1,\dots,m,
	 \end{equation}for a positive definite kernel $\bb{G}$ with non-negative entries and some measurable positive functions $\nu_i$, $i=1,\dots,m$, $t>0$. Choosing $d\mu(t)=t^{-1}\exp(-1/(4t))dt$, and following the calculations in \citet{kleiber2012nonstationary} shows that the kernel $\bb{C}: \RR^d \times \RR^d \rightarrow \RR^{m \times m}$, defined via
	 \begin{eqnarray*}
	C_{ij}(\x,\y) &=& \displaystyle{\frac{\lvert\bb{\Sigma_i}(\x)\lvert^{1/4}\lvert\bb{\Sigma_j}(\y)\lvert^{1/4}}{\lvert \bb{\Sigma_{ij}}(\x,\y)\rvert^{1/2}}2^{\nu_i(\x)/2+\nu_j(\y)/2}\times} \\ &&\displaystyle{M_{\nu_i(\x)/2+\nu_j(\y)/2}\left(\sqrt{(\x-\y)^\top\bb{\Sigma_{ij}}(\x,\y)^{-1}(\x-\y)+G_{ij}(\x,\y)} \right)}, 
	 \end{eqnarray*} $i,j=1,\dots,m$, is positive definite. Here, $M_\nu$ denotes the Whittle-Mat\'ern covariance function with smoothness parameter $\nu$ \citep{matern1986}. For $G_{ij}(\x,\y) \equiv 0$, $i,j=1,\dots,m$, we essentially recover the non-stationary Whittle-Mat\'{e}rn model in \citet{kleiber2012nonstationary}.
\end{example}

Example~\ref{ex:maternnonstat} has a nice interpretation regarding the use of the more general conditionally negative definite kernel~\eqref{eq:maternnonstatcnd} in comparison with a solely additive separable structure. Here, the additive separable part of the conditionally negative definite matrix-valued kernel acts on the parameters of the underlying covariance function, whereas the positive definite part only affects the measurement of the distances between two locations $\x$ and $\y$, by augmenting the Mahalanobis distance. The same effect occurs when adopting the non-stationary Cauchy model in \citet[Equation 18]{kleiber2015}. Since the conditionally negative definite kernel~\eqref{eq:maternnonstatcnd} should be a general one to recover a Mat\'{e}rn-type model in the framework of Proposition~\ref{paciorekconstruction}, cf.\,Theorem~\ref{thm:cndstructure}, the additive separable structure $(\nu_i(\x)+\nu_j(\y))/2$ seems to be the only option to model the smoothness with the approach above.

Additive separable conditionally negative definite kernels are also used in Corollary~2 in \citet{alsultan2019}. The latter is based on Theorem~2 in \citet{alsultan2019}, which can be concretized via Theorem~\ref{thm:schoenberg}. In fact, the matrix-valued kernels $\bb{g}=(g_{ij})_{i,j=1,\dots,m}$ given there are exactly the positive conditionally negative definite matrix-valued kernels as introduced in Section~\ref{sec:prelim}. Similarly, various other cross-covariances in \citet{ma2013mittag} and \citet{balakrishnan2015}, for instance, can be generalized by replacing the variogram used there with a pseudo cross-variogram. 

\section{Derivative Related Results}

In this section, we present neat constructions of positive definite matrix-valued functions involving derivatives, which can also be transferred to the multivariate case via pseudo cross-variograms. Our central tool here is Lemma \ref{lem:kernel}, which enables us to prove the following multivariate versions of Corollaries 3.12.8 and 3.12.9 in \citet{sasvari2013multivariate}, cf.\,also \citet{gneiting2001}, and Corollary~3 in \citet{ma2005spatio}. 

\begin{prop}\label{secondderiv}
	Let $\G: \RR^d \rightarrow \RR^{m \times m}$ be a pseudo cross-variogram with twice continuously differentiable component functions $\gamma_{ij}$, $i,j=1,\dots,m$. Then the function $\C: \RR^d \rightarrow \RR^{m \times m}$, defined via
	
	$$ C_{ij}(\h)= \frac{\partial^{2}}{\partial h_k^2}\gamma_{ij}(\h), \quad \h\in \RR^d, \quad i,j=1,\dots,m, $$ is positive definite for $k=1,\dots,d$.
\end{prop}

\begin{proof} We follow the univariate proof in \citet{sasvari2013multivariate}. Let $\{\bb{e_1},\dots,\bb{e_d}\}$ be the canonical basis in $\RR^d$. Then we have \begin{align*}
	\C(\h)= \lim_{\varepsilon \rightarrow 0}\frac{\G(\h + \varepsilon \bb{e_k}) + \G(\h-\varepsilon \bb{e_k}) - 2\G(\h)}{\varepsilon^2}. 
	\end{align*}
	Due to Lemma~\ref{lem:kernel}, $\C$ is positive definite as pointwise limit of positive definite functions.
\end{proof}

\begin{cor}
	Let $\mathcal{L}$ be the continuous extension of a completely monotone function on $[0,\infty)$. Let $\G: \RR^d \times \RR \rightarrow \RR^{m \times m}$ be a pseudo cross-variogram of a spatio-temporal random field with twice continuously differentiable component functions. Then, the function $\C: \RR^d \times \RR \rightarrow \RR^{m \times m}$, defined via $$C_{ij}(\h,u) = \mathcal{L}(\gamma_{ij}(\h,u)) \frac{\partial^{2}}{\partial u^2}\gamma_{ij}(\h,u)+ \mathcal{L}^\prime(\gamma_{ij}(\h,u))\left(\frac{\partial}{\partial u}\gamma_{ij}(\h,u)\right)^2, \quad i,j=1,\dots,m, $$ is positive definite.
\end{cor}

\begin{proof}
	We follow the arguments in \citet{ma2005spatio}. Since Bernstein functions vanishing at zero also operate on pseudo cross-variograms \citep{berg1984harmonic, doerr2021characterization}, the matrix-valued function $$(\h,u) \mapsto \left(\int_0^{\gamma_{ij}(\h,u)} \mathcal{L}(y) dy\right)_{i,j=1,\dots,m},$$
	is again a pseudo cross-variogram. Differentiating with respect to $u$ and applying Proposition~\ref{secondderiv} then gives the result.
\end{proof}

\begin{prop} \label{deriv_sasvari} 
	Let $\G$ be an isotropic pseudo cross-variogram in $\RR^{d+1}$ with $\gamma_{ij}(\h)= g_{ij}(\lVert \h \rVert^2)$, $i,j=1,\dots,m$, where $\bb{g}: [0,\infty) \rightarrow \RR^{m \times m}$ is a continuous matrix-valued function with twice continuously differentiable component functions in $[0,\infty)$. Then the function $\C: \RR^d \rightarrow \RR^{m\times m}$, defined via $$C_{ij}(\h)=g_{ij}^\prime\left(\lVert\h\rVert^2\right), \quad  i,j=1,\dots,m,$$ is positive definite on $\RR^d$. 
\end{prop}

\begin{proof} We follow the univariate proof in \citet{sasvari2013multivariate}.
	Let $\varepsilon \in \RR$. According to Lemma~\ref{lem:kernel}, the function $$ \h \mapsto \left(g_{ij}\left(\lVert \h+ \varepsilon \bb{e_{d+1}}  \rVert^2\right)+ g_{ij}\left(\lVert \h -  \varepsilon \bb{e_{d+1}} \rVert^2\right) -2 g_{ij}\left(\lVert \h\rVert^2\right)\right)_{i,j=1,\dots,m},$$ 
	is a positive definite matrix-valued function; so is its restriction to $\RR^d$, which reads 
	$$\bb{\tilde h} \mapsto  2\left(g_{ij}\left(\lVert \bb{\tilde h}\rVert^2 + \varepsilon^2\right) -g_{ij}\left(\lVert \bb{\tilde h}\rVert^2\right)\right)_{i,j=1,\dots,m}.$$ The assertion follows by dividing by $\varepsilon^2$ and letting $\epsilon \rightarrow 0$.  
\end{proof}

\section{Infinite Divisibility}

As pointed out in \citet{berg1984harmonic}, Schoenberg's theorem relates conditionally negative definite functions to so-called infinitely divisible positive definite functions. In accordance to Definition~3.2.6 in \citet{berg1984harmonic}, a matrix-valued positive definite function $\C: \RR^d \rightarrow \RR^{m \times m}$ is infinitely divisible, if for each $n\in \NN$, there exists a positive definite matrix-valued function $\C_n$ such that $\C=\C_n^{\ast n}$, which is equivalent to $\C^{\ast r}$ being positive definite for all $r >0$, see Proposition~3.2.7 in \citet{berg1984harmonic}. As a particular example, we extend a construction in \citet{kosaki2008} to the multivariate case which provides a variety of different infinitely divisible positive definite matrix-valued functions. 

\begin{lemma} \label{lem:kosaki} 
	Let $\G: \RR^d \rightarrow \RR^{m \times m}$ be a pseudo cross-variogram. Let $a>0$, $b\ge 0$. Then the function $\C: \RR^d \rightarrow \RR^{m\times m}$, defined via
	$$C_{ij}(\h) = \frac{1+b\gamma_{ij}(\h)}{1+ a\gamma_{ij}(\h)},  \quad i,j=1,\dots,m,$$ is infinitely divisible if and only if $a \ge b$.
\end{lemma}

\begin{proof}
	We follow the proof of the univariate version in \citet{kosaki2008}. Assume first that $\C$ is infinitely divisible. Then $\C$ is positive definite and, thus, has to fulfill the inequality $C_{ij}(h)^2 \le C_{ii}(0)C_{jj}(0)$, implying $a \ge b$. For the if part, it suffices to show that $\C^{\ast r}$ is positive definite for $r \in (0,1)$ due to the stability of positive definite matrix-valued functions under Schur products and pointwise limits.  
	For $r \in (0,1)$, we have 
	$$x^r= \frac{\sin(\pi r)}{\pi}\int_0^\infty \frac{x}{x + \lambda} \lambda^{r-1} d\lambda, \quad x \ge 0.$$ We thus have 
	$$\left(\frac{1+b\gamma_{ij}(\h)}{1+ a\gamma_{ij}(\h)}\right)^r = \frac{\sin(\pi r)}{\pi}\int_0^\infty \left(\frac{b}{a\lambda+b} +\frac{\lambda(a-b)}{a\lambda+b}\cdot \frac1{1+ \lambda + (a\lambda+b)\gamma_{ij}(\h)}\right)\lambda^{r-1} d\lambda.   $$
	Since $\lambda$ and $a\lambda+b$ are non-negative, and the set of conditionally negative definite functions forms a convex cone, the function $$\h \mapsto \left(\lambda + (a\lambda+b)\gamma_{ij}(\h)\right)_{i,j=1,\dots,m},$$ is conditionally negative definite. Hence, the function $$\h \mapsto \left(\frac1{1+ \lambda + (a\lambda+b)\gamma_{ij}(\h)}\right)_{i,j=1,\dots,m},  $$ is positive definite, cf.\,\citet{berg1984harmonic} or Corollary~3.6 in \citet{doerr2021characterization}. Since $\frac{b}{a\lambda+b}$ and $\frac{\lambda(a-b)}{a\lambda+b}$ are non-negative, and the set of positive definite matrix-valued functions also forms a convex cone, the function $$\h \mapsto \left(\frac{b}{a\lambda+b} +\frac{\lambda(a-b)}{a\lambda+b}\cdot \frac1{1+ \lambda + (a\lambda+b)\gamma_{ij}(\h)}\right)_{i,j=1,\dots,m},  $$ is positive definite for all $\lambda >0$. Consequently, $\C^{\ast r}$ is positive definite as a mixture of positive definite functions.
\end{proof}

	\begin{theorem}\label{thm:MultKosaki}
		Let $\G: \RR^d \rightarrow \RR^{m\times m}$ be a pseudo cross-variogram. Let $f$ be an entire function taking real values for the reals. Assume that
		\begin{itemize}
			\item $f(0)>0$ and $f^\prime(0)=0$;
			\item all the zeros of f are purely imaginary;
			\item the order $\rho$ of $f$ is less than 2, i.e.,\, $$ \rho = \limsup_{r \rightarrow \infty} \frac{\log\log M(r)}{\log r} <2 $$ with $M(r)=\max\{\lvert f(z)\rvert; \lvert z \rvert=r \}. $
		\end{itemize}
		Then the function $\C: \RR^d \rightarrow \RR^{m \times m}$, defined by $$C_{ij}(\h)=\frac{f(\nu \sqrt{\gamma_{ij}(\h)})}{f(\sqrt{\gamma_{ij}(\h)})}, \quad i,j=1,\dots,m,$$ is infinitely divisible for $\nu \in [0,1]$. 
	\end{theorem}

\begin{proof}
	The proof of Theorem~2 in \citet{kosaki2008} shows that $\C$ can be written as
	$$C_{ij}(\h)= \lim_{k \rightarrow \infty} \prod_{n=1}^k \left(\frac{1+\nu^2\gamma_{ij}(\h)/\alpha_n^2}{1+\gamma_{ij}(\h)/\alpha_n^2}\right), \quad \h \in \RR^d, \quad i,j=1,\dots,m, $$
	for some $\alpha_n >0$. Since componentwise products and pointwise limits of infinitely divisible positive definite matrix-valued functions are infinitely divisible again, the assertion follows due to Lemma~\ref{lem:kosaki}.
\end{proof}

\begin{example}
	The function $t \mapsto \cosh(t)$ fulfills the conditions in Theorem~\ref{thm:MultKosaki} \citep[Remark 4]{kosaki2008}. Hence, the function $\C: \RR^d \rightarrow \RR^{m \times m}$
	\begin{equation} \label{eq:buellmult} C_{ij}(\h) = \frac{\cosh(\nu \sqrt{\gamma_{ij}(\h))}}{\cosh(\sqrt{\gamma_{ij}(\h))}}, \quad i,j=1,\dots,m, 
	\end{equation} is a valid cross-covariance function for $\nu \in [0,1]$. 
	Using a general non-negative conditionally negative definite kernel instead of a pseudo cross-variogram in Equation~\eqref{eq:buellmult} shows that model~\eqref{eq:buellmult} generalizes model~(10) in \citet{ma2013mittag}.  
	For $\nu=0$, we have $$C_{ij}(\h) = \frac{1}{\cosh(\sqrt{\gamma_{ij}(\h)})},\quad i,j=1,\dots,m, $$ which generalizes a univariate covariance model proposed in a meterological context \citep{buell1972correlation, gneiting1999correlation}. More examples of functions which can be used in Theorem~\ref{thm:MultKosaki} can be found in \citet{kosaki2008}. 
\end{example}

\section*{Acknowledgment}

The authors gratefully acknowledge support by the German Research Foundation (DFG) through the Research Training Group RTG 1953.


\begin{thebibliography}{10}
\bibitem[Allard et~al., 2022]{allard2022fully}
 {\sc Allard, D., Clarotto, L., and Emery, X.}(2022). 
 Fully nonseparable Gneiting covariance functions for multivariate space-time data. \emph{HAL:03564931}. \url{https://hal.archives-ouvertes.fr/hal-03564931/}.

\bibitem[Alsultan and Ma, 2019]{alsultan2019}
{\sc Alsultan, R. and Ma, C.}
(2019). K-differenced vector random fields. \emph{Theory Probab. Appl.} {\bf 63,} 393--407. \url{https://doi.org/10.1137/S0040585X97T989131}.

\bibitem[Apanasovich and Genton, 2010]{apanasovich2010}
{\sc Apanasovich, T. and Genton, M.}
(2010). Cross-covariance functions for multivariate random fields based on latent dimensions. \emph{Biometrika} {\bf 97,} 15--30. \url{https://doi.org/10.1093/biomet/asp078}.

\bibitem[Apanasovich et~al., 2012]{apanasovich2012}
{\sc Apanasovich, T., Genton, M. and Sun, Y.}
(2012). A valid Mat{\'e}rn class of cross-covariance functions for multivariate random fields with any number of components. \emph{J. Amer. Statist. Assoc.} {\bf 107,} 180--193. \url{https://doi.org/10.1080/01621459.2011.643197}.	

\bibitem[Arroyo and Emery, 2017]{arroyo2017spectral}
{\sc Arroyo, D. and Emery, X.}
(2017). Spectral simulation of vector random fields with stationary Gaussian increments in d-dimensional Euclidean spaces. \emph{Stoch. Environ. Res. Risk Assess.} {\bf 31,} 1583--1592. \url{https://doi.org/10.1007/s00477-016-1225-7}.	

\bibitem[Balakrishnan et~al., 2015]{balakrishnan2015}
{\sc Balakrishnan, N., Ma, C. and Wang, R.}
(2015). Logistic vector random fields with logistic direct and cross covariances. \emph{J. Statist. Plann. Inference} {\bf 161,} 109--118. \url{https://doi.org/10.1016/j.jspi.2015.01.004}.

\bibitem[Berg et~al., 1984]{berg1984harmonic}
{\sc Berg, C., Christensen, J. P. R. and Ressel, P.}
(1984). \emph{Harmonic Analysis on Semigroups: Theory of Positive Definite and Related Functions},
Springer, New York. \url{https://doi.org/10.1007/978-1-4612-1128-0}.

\bibitem[Bourotte et~al., 2016]{bourotte2016}
{\sc Bourotte, M., Allard, D. and Porcu, E.}
(2016). A flexible class of non-separable cross-covariance functions for multivariate space--time data. \emph{Spat. Stat.} {\bf 18,} 125--146. \url{https://doi.org/10.1016/j.spasta.2016.02.004}.

\bibitem[Boyd and Vandenberghe, 2009]{BoydStephenP2009Co}
{\sc }
(2009). \emph{Convex optimization}, 7th edn. with corrections. Cambridge University Press, Cambridge. \url{https://doi.org/10.1017/CBO9780511804441}.

\bibitem[Buell, 1972]{buell1972correlation}
{\sc Buell, C.\,E.}
(1972). Correlation functions for wind and geopotential on isobaric surfaces. \emph{J. Appl. Meteorol.} {\bf 11,} 51--59. 
\url{https://doi.org/10.1175/1520-0450(1972)011<0051:CFFWAG>2.0.CO;2}.

\bibitem[Chen and Genton, 2019]{chen2019parametric}
{\sc Chen, W. and Genton, M.}
(2019). Parametric variogram matrices incorporating both bounded and unbounded functions. \emph{Stoch. Environ. Res. Risk Assess.} {\bf 33,} 1669--1679. \url{https://doi.org/10.1007/s00477-019-01710-1}.

\bibitem[Chen et~al., 2021]{chen2021space}
{\sc Chen, W. and Genton, M. and Sun, Y.}
(2021). Space-time covariance structures and models. \emph{Annu. Rev. Stat. Appl.} {\bf 8,} 191--215. \url{https://doi.org/10.1146/annurev-statistics-042720-115603}.

\bibitem[Cressie and Zammit-Mangion, 2016]{cressie2016}
{\sc Cressie, N. and Zammit-Mangion, A.}
(2016). {M}ultivariate spatial covariance models: a conditional approach. \emph{Biometrika.} {\bf 103,} 915--935.
\url{https://doi.org/10.1093/biomet/asw045}.

\bibitem[D\"orr and Schlather, 2021]{doerr2021characterization}
{\sc D{\"o}rr, C. and Schlather, M.}
(2021). Characterization theorems for pseudo-variograms. \emph{arXiv:2112.02595}.  	
\url{https://doi.org/10.48550/arXiv.2112.02595}.


\bibitem[Du and Ma, 2012]{du2012variogram}
{\sc Du, J. and Ma, C.}
(2012). Variogram matrix functions for vector random fields with second-order increments. \emph{Math. Geosci.} {\bf 44,} 411--425. \url{https://doi.org/10.1007/s11004-011-9377-y}.

\bibitem[Fonseca and Steel, 2011]{fonseca2011general}
{\sc Fonseca, T. and Steel, M.}
(2011). A general class of nonseparable space--time covariance models. \emph{Environmetrics.} {\bf 22,} 224--242.  \url{https://doi.org/10.1002/env.1047}.

\bibitem[Genton et~al., 2015]{genton2015maxstable}
{\sc Genton, M., Padoan, S. and Sang, H.}
(2015). Multivariate max-stable spatial processes. \emph{Biometrika.} {\bf 102,} 215--230. \url{https://doi.org/10.1093/biomet/asu066}.

\bibitem[Genton and Kleiber, 2015]{genton2015}
{\sc Genton, M. and Kleiber, W.}
(2015). {C}ross-covariance functions for multivariate geostatistics. \emph{Statist. Sci.} {\bf 30,} 147--163. \url{https://doi.org/10.1214/14-STS487}.

\bibitem[Gneiting, 1999]{gneiting1999correlation}
{\sc Gneiting, T.}
(1999). {C}orrelation functions for atmospheric data analysis. \emph{Q. J. R. Meteorol.} {\bf 125,} 2449--2464. \url{https://doi.org/10.1002/qj.49712555906}.

\bibitem[Gneiting, 2002a]{gneiting2002compactly}
{\sc Gneiting, T.}
(2002). {C}ompactly supported correlation functions. \emph{J. Multivariate Anal.} {\bf 83,} 493--508. \url{https://doi.org/10.1006/jmva.2001.2056}.

\bibitem[Gneiting, 2002b]{gneiting2002}
{\sc Gneiting, T.}
(2002). {N}onseparable, stationary covariance functions for space-time data. \emph{J. Amer. Statist. Assoc.} {\bf 97,} 590--600. \url{https://doi.org/10.1198/016214502760047113}.

\bibitem[Gneiting et~al., 2010]{gneiting2010}
{\sc Gneiting, T., Kleiber, W. and Schlather, M.}
(2010). {M}at{\'e}rn cross-covariance functions for multivariate random fields. \emph{J. Amer. Statist. Assoc.} {\bf 105,} 1167--1177. \url{https://doi.org/10.1198/jasa.2010.tm09420}.

\bibitem[Gneiting et~al., 2001]{gneiting2001}
{\sc Gneiting, T., Sasv{\'a}ri, Z. and Schlather, M.}
(2001). Analogies and correspondences between variograms and covariance functions. \emph{Adv. Appl. Prob.} {\bf 33,} 617--630. \url{https://doi.org/10.1239/aap/1005091356}.

\bibitem[Goulard and Voltz, 1992]{goulard1992}
{\sc Goulard, M. and Voltz, M.}
(1992). {L}inear coregionalization model: tools for estimation and choice of cross-variogram matrix. \emph{Math. Geol.} {\bf 24,} 269--286. \url{https://doi.org/10.1007/BF00893750}.


\bibitem[Gradshteyn and Ryzhik, 2000]{gradshteyn2014}
{\sc Gradshteyn, I. and Ryzhik, I.}
(2000). \emph{{T}able of {I}ntegrals, {S}eries, and {P}roducts}, 6th edn.
Academic Press, San Diego. \url{https://doi.org/10.1016/B978-0-12-294757-5.X5000-4}.


\bibitem[Kleiber and Nychka, 2012]{kleiber2012nonstationary}
{\sc Kleiber, W. and Nychka, D.}
(2012). Nonstationary modeling for multivariate spatial processes. \emph{J. Multivariate Anal.} {\bf 112,} 76--91. \url{https://doi.org/10.1016/j.jmva.2012.05.011}.

\bibitem[Kleiber and Porcu, 2015]{kleiber2015}
{\sc Kleiber, W. and Porcu, E.}
(2015). {N}onstationary matrix covariances: compact support, long range dependence and quasi-arithmetic constructions. \emph{Stoch. Environ. Res. Risk Assess.} {\bf 29,} 193--204. \url{https://doi.org/10.1007/s00477-014-0867-6}.

\bibitem[Kosaki, 2008]{kosaki2008}
{\sc Kosaki, H.}
(2008). On infinite divisibility of positive definite functions arising from operator means. \emph{J. Funct. Anal.} {\bf 254,} 84--108. \url{https://doi.org/10.1016/j.jfa.2007.09.021}.

\bibitem[Li et~al., 2008]{li2008}
{\sc Li, B., Genton, M. and Sherman, M.}
(2008). {T}esting the covariance structure of multivariate random fields. \emph{Biometrika.} {\bf 95,} 813--829. \url{https://doi.org/10.1093/biomet/asn053}.

\bibitem[Li and Zhang, 2011]{li2011}
{\sc Li, B. and Zhang, H.}
(2011). {A}n approach to modeling asymmetric multivariate spatial covariance structures. \emph{J. Multivariate Anal.} {\bf 102,} 1445--1453. \url{https://doi.org/10.1016/j.jmva.2011.05.010}.

\bibitem[Ma, 2003a]{ma2003spatio}
{\sc Ma, C.}
(2003a). Spatio-temporal stationary covariance models. \emph{J. Multivariate Anal.} {\bf 86,} 97--107. \url{https://doi.org/10.1016/S0047-259X(02)00014-3}.

\bibitem[Ma, 2003b]{ma2003families}
{\sc Ma, C.}
(2003b). Families of spatio-temporal stationary covariance models. \emph{J. Statist. Plann. Inference.} {\bf 116,} 489--501. \url{https://doi.org/10.1016/S0378-3758(02)00353-1}.


\bibitem[Ma, 2004]{ma2004use}
{\sc Ma, C.}
(2004). The use of the variogram in construction of stationary time series models. \emph{J. Appl. Probab.} {\bf 41,} 1093--1103. \url{https://doi.org/10.1239/jap/1101840554}.

\bibitem[Ma, 2005]{ma2005spatio}
{\sc Ma, C.}
(2005). Spatio-temporal variograms and covariance models. \emph{Adv. Appl. Prob.} {\bf 37,} 706--725. \url{https://doi.org/10.1239/aap/1127483743}.

\bibitem[Ma, 2011a]{ma2011class}
{\sc Ma, C.}
(2011). A class of variogram matrices for vector random fields in space and/or time. \emph{Math. Geosci.} {\bf 43,} 229--242. \url{https://doi.org/10.1007/s11004-010-9310-9}.

\bibitem[Ma, 2011b]{ma2011vector}
{\sc Ma, C.}
(2011). Vector random fields with second-order moments or second-order increments. \emph{Stoch. Anal. Appl.} {\bf 29,} 197--215. \url{https://doi.org/10.1080/07362994.2011.532039}.

\bibitem[Ma, 2013]{ma2013mittag}
{\sc Ma, C.}
(2013). {M}ittag-{L}effler vector  random fields with {M}ittag-{L}effler direct and cross covariance functions. \emph{Ann. Inst. Statist. Math.} {\bf 65,} 941--958. \url{https://doi.org/10.1007/s10463-013-0398-9}.

\bibitem[Majumdar and Gelfand, 2007]{majumdar2007}
{\sc Majumdar, A. and Gelfand, A.}
(2007). {M}ultivariate spatial modeling for geostatistical data using convolved covariance functions. \emph{Math. Geol.} {\bf 39,} 225--245. \url{https://doi.org/10.1007/s11004-006-9072-6}.

\bibitem[Mat{\'e}rn, 1986]{matern1986}
{\sc Mat{\'e}rn, B.}
(1986). \emph{Spatial variation},
Springer, Berlin. \url{https://doi.org/10.1007/978-1-4615-7892-5}.

\bibitem[Moreva and Schlather, 2016]{moreva2016}
{\sc Moreva, O. and Schlather, M.}
(2016). {B}ivariate covariance functions of {P}\'{o}lya type. \emph{arXiv:1609.06561}. \url{https://arxiv.org/abs/1609.06561}.

\bibitem[Myers, 1982]{myers1982matrix}
{\sc Myers, D. E.}
(1982). Matrix formulation of co-kriging. \emph{J. Int. Assoc. Math. Geol.} {\bf 14,} 249--257. \url{https://doi.org/10.1007/BF01032887}.

\bibitem[Oesting et~al., 2017]{oesting2017}
{\sc Oesting, M., Schlather, M. and Friederichs, P.}
(2017). Statistical post-processing of forecasts for extremes using bivariate Brown-Resnick processes with an application to wind gusts. \emph{Extremes.} {\bf 20,} 309--332. \url{https://doi.org/10.1007/s10687-016-0277-x}.


\bibitem[Papritz et~al., 1993]{papritz1993}
{\sc Papritz, A., K{\"u}nsch, H.~R. and Webster, R.}
(1993). On the pseudo cross-variogram. \emph{Math. Geol.} {\bf 25,} 1015--1026. \url{https://doi.org/10.1007/BF00911547}.

\bibitem[Porcu et~al., 2022]{porcu2022criteria}
{\sc Porcu, E., Emery, X. and Mery, N.}
(2022). Criteria and characterizations for spatially isotropic and temporally symmetric matrix-valued covariance functions. \emph{HAL:03516047}. \url{https://hal.archives-ouvertes.fr/hal-03516047/}.

\bibitem[Porcu et~al., 2006]{porcu2006nonseparable}
{\sc Porcu, E., Gregori, P. and Mateu, J.}
(2006). Nonseparable stationary anisotropic space--time covariance functions. \emph{Stoch. Environ. Res. Risk Assess.} {\bf 21,} 113--122. \url{https://doi.org/10.1007/s00477-006-0048-3}.


\bibitem[Porcu et~al., 2007]{porcu2007covariance}
{\sc Porcu, E., Mateu, J. and Bevilacqua, M.}
(2007). Covariance functions that are stationary or nonstationary in space and stationary in time. \emph{Stat. Neerl.} {\bf 61,} 358--382. \url{https://doi.org/10.1111/j.1467-9574.2007.00364.x}.

\bibitem[Porcu and Schilling, 2011]{porcu2011}
{\sc Porcu, E. and Schilling, R. L.}
(2011). From {S}choenberg to {P}ick-{N}evanlinna: Toward a complete picture of the variogram class. \emph{Bernoulli.} {\bf 17,} 441--455. \url{https://doi.org/10.3150/10-BEJ277}.

\bibitem[Porcu and Zastavnyi, 2011]{porcu2011characterization}
{\sc Porcu, E. and Zastavnyi, V.}
(2011). Characterization theorems for some classes of covariance functions associated to vector valued random fields. \emph{J. Multivariate Anal.} {\bf 102,} 1293--1301. \url{https://doi.org/10.1016/j.jmva.2011.04.013}.

\bibitem[Porcu et~al., 2018]{porcu2018shkarofsky}
{\sc Porcu, E., Bevilacqua, M. and Hering, A.}
(2018). {T}he {S}hkarofsky-{G}neiting class of covariance models for bivariate {G}aussian random fields. \emph{Stat.} {\bf 7,} e207. \url{https://doi.org/10.1002/sta4.207}.

\bibitem[Qadir et~al., 2020]{qadir2020flexible}
{\sc Qadir, G. A., Eu{\'a}n, C. and Sun, Y.}
(2020). {F}lexible modeling of variable asymmetries in cross-covariance functions for multivariate random fields. \emph{J. Agricul. Biol. and Environ. Stat.} {\bf 26,} 1--22. \url{https://doi.org/10.1007/s13253-020-00414-2}.

\bibitem[Qadir et~al., 2021]{qadir2021semiparametric}
{\sc Qadir, G. A. and Sun, Y.}
(2021). Semiparametric estimation of cross-covariance functions for multivariate random fields. \emph{Biometrics.} {\bf 77,} 547--560. \url{https://doi.org/10.1111/biom.13323}.

\bibitem[Salvana and Genton, 2020]{salvana2020}
{\sc Salvana, M. and Genton, M.}
(2020). {N}onstationary cross-covariance functions for multivariate spatio-temporal random fields. \emph{Spat. Stat.} {\bf 37,} 547--560. \url{https://doi.org/10.1016/j.spasta.2020.100411}.

\bibitem[Sampson and Guttorp, 1992]{sampson1992}
{\sc Sampson, P.and Guttorp, P.}
(1992). {N}onparametric estimation of nonstationary spatial covariance structure. \emph{J. Amer. Statist. Assoc.} {\bf 87,} 108--119. \url{https://doi.org/10.2307/2290458}.

\bibitem[Sasv{\'a}ri, 2013]{sasvari2013multivariate}
{\sc Sasv{\'a}ri, Z.}
(2013). \emph{Multivariate characteristic and correlation functions},
De Gruyter, Berlin. \url{https://doi.org/10.1515/9783110223996}.

\bibitem[Schilling et~al., 2012]{schilling2012}
{\sc Schilling, R. L. and Song, R. and Vondracek, Z.}
(2012). \emph{Bernstein functions}, 2nd rev. edn. 
De Gruyter, Berlin. \url{https://doi.org/10.1515/9783110269338}.

\bibitem[Schlather, 2010]{schlather2010}
{\sc Schlather, M.}
(2010). {S}ome covariance models based on normal scale mixtures. \emph{Bernoulli.} {\bf 16,} 780--797. \url{https://doi.org/10.3150/09-BEJ226}.

\bibitem[Schoenberg, 1938]{schoenberg1938}
{\sc Schoenberg, I.}
(1938). {M}etric spaces and completely monotone functions. \emph{Ann. Math.} {\bf 16,} 811--841. \url{https://doi.org/10.2307/1968466}.


\bibitem[Stein, 2005]{stein2005jasa}
{\sc Stein, M.}
(2005). {S}pace-time covariance functions. \emph{J. Amer. Statist. Assoc.} {\bf 100,} 310--321. \url{https://doi.org/10.1198/016214504000000854}.

\bibitem[Ver Hoef and Barry, 1998]{ver1998}
{\sc Ver Hoef, J. and Barry, R.}
(1998). {C}onstructing and fitting models for cokriging and multivariable spatial prediction. \emph{J. Statist. Plann. Inference.} {\bf 69,} 275--294. \url{https://doi.org/10.1016/S0378-3758(97)00162-6}.

\bibitem[Vu et~al., 2021]{vu2021}
{\sc Vu, Q., Zammit-Mangion, A. and Cressie, N.}
(2021). Modeling nonstationarity and asymmetric multivariate spatial covariances via deformation. \emph{Statistica Sinica.} In press.

\bibitem[Wackernagel, 2003]{wackernagel2003multivariate}
{\sc Wackernagel, H.}
(2003). \emph{Multivariate Geostatistics: an introduction with applications}, 3rd completely rev. edn.
Springer, Berlin. \url{https://doi.org/10.1007/978-3-662-05294-5}.

\end{thebibliography}
\end{document}